\documentclass[a4paper,12pt, reqno]{amsart}
\usepackage{a4}
\usepackage{amssymb}
\usepackage{amsmath}
\usepackage{amsthm}
\usepackage{amstext}
\usepackage{amscd}
\usepackage{latexsym}
\usepackage{mathabx}
\usepackage{graphics}
\usepackage{color}
\usepackage[all]{xy}
\usepackage{verbatim}
\usepackage{textcomp}
\usepackage{color}
\usepackage{soul}

\newtheorem{theorem}{Theorem}
\newtheorem{lemma}{Lemma}

\newtheorem{corollary}{Corollary}

\newcommand{\GL}{{\rm GL}}
\newcommand{\SL}{{\rm SL}}
\newcommand{\PD}{{\rm PD}}
\newcommand{\ind}{{\rm ind}}

\title[Multiplicity formula]{Multiplicity formula for restriction of  representations of $\widetilde{\GL_{2}}(E)$ to $\widetilde{\SL_{2}}(E)$}
\author{Shiv Prakash Patel}

\address{School of Mathematics, Tata Institute of Fundamental
Research, Homi Bhabha Road, Colaba, Mumbai 400005, India}

\email{shiv@math.tifr.res.in}

\author{Dipendra Prasad}

\address{School of Mathematics, Tata Institute of Fundamental
Research, Homi Bhabha Road, Colaba, Mumbai 400005, India}

\email{dprasad@math.tifr.res.in}


\date{\today}

\begin{document}

\begin{abstract}
In this note we prove a certain multiplicity formula regarding the restriction of an irreducible admissible genuine representation of a 2-fold cover $\widetilde{\GL}_{2}(E)$ of $\GL_{2}(E)$ to the 2-fold cover $\widetilde{\SL}_{2}(E)$ of $\SL_{2}(E)$, and find in particular that this multiplicity may not be one, a result that seems to have been noticed before. The proofs follow the standard path via Waldspurger's analysis of theta correspondence between $\widetilde{\SL}_{2}(E)$ and ${\rm PGL}_{2}(E)$.
\end{abstract}
\maketitle
\section{Introduction}
This paper will be concerned with certain 2-fold covers of $\GL_{2}(E)$ to be called the metaplectic covering of $\GL_{2}(E)$, where $E$ is a non-Archimedian local field. We recall that there is a unique (up to isomorphism) non-trivial 2-fold cover of $\SL_{2}(E)$ called the metaplectic cover and denoted by $\widetilde{\SL}_{2}(E)$ in this paper, but there are many inequivalent 2-fold coverings of $\GL_{2}(E)$ which extend this 2-fold covering of $\SL_{2}(E)$.  We fix a covering of $\GL_{2}(E)$ as follows. Observe that $\GL_{2}(E)$ is the semi-direct product of $\SL_{2}(E)$ and $E^{\times}$, where $E^{\times}$ sits inside $\GL_{2}(E)$ as $e \mapsto \left( \begin{matrix} e & 0 \\ 0 & 1 \end{matrix} \right)$. This action of $E^{\times}$ on $\SL_{2}(E)$ lifts uniquely to an action of $E^{\times}$ on  $\widetilde{\SL}_{2}(E)$. Denote $\widetilde{\GL}_{2}(E) = \widetilde{\SL}_{2}(E)  \rtimes E^{\times}$ and call this the metaplectic cover of $\GL_{2}(E)$. Thus the metaplectic cover of $\GL_{2}(E)$ that we consider in this paper is that cover of $\GL_{2}(E)$ which extends the metaplectic cover of $\SL_{2}(E)$ and is further split on the subgroup $\left\{ \left( \begin{matrix} e & 0 \\ 0 & 1 \end{matrix} \right) : e \in E^{\times} \right\}$. Moreover, we have the following short exact sequence of locally compact topological groups
\[
1 \longrightarrow \{ \pm 1 \} \rightarrow \widetilde{\GL}_{2}(E) \xrightarrow{p} \GL_{2}(E) \rightarrow 1.
\]

For any subset $X$ of $\GL_{2}(E)$ we write $\tilde{X}$ for its inverse image in $\widetilde{\GL}_{2}(E)$. Let $Z$ be the center of $\GL_{2}(E)$ which we identify with $E^{\times}$. It can be checked that $\tilde{Z}$ is an abelian subgroup of $\widetilde{\GL}_{2}(E)$ but is not the center of $\widetilde{\GL}_{2}(E)$; the center of $\widetilde{\GL}_{2}(E)$ is $\tilde{Z^2}$. The centralizer of $\widetilde{\SL}_{2}(E)$ inside $\widetilde{\GL}_{2}(E)$ is $\tilde{Z}$. Let $\widetilde{\GL}_{2}(E)_{+} = \tilde{Z} \cdot \widetilde{\SL}_{2}(E)$. Let $\mu$ be a genuine character of $\tilde{Z}$ and $\tau$ an irreducible admissible genuine representations of $\widetilde{\SL}_{2}(E)$. We say that $\mu$ and $\tau$ are compatible if $\mu|_{\widetilde{ \{ \pm 1 \} }} = \omega_{\tau}$ where $\omega_{\tau}$ is the central character of $\tau$ and if so, we define a representation $\mu \tau$ of $\widetilde{\GL}_{2}(E)_{+}$ whose restriction to $\widetilde{\SL}_{2}(E)$ is $\tau$ and central character is $\mu$. One may choose the representatives of the quotient $\widetilde{\GL}_{2}(E)/\widetilde{\GL}_{2}(E)_{+} \cong E^{\times}/E^{\times 2}$ to be $g(a) := \left( \begin{matrix} a & 0 \\ 0 & 1 \end{matrix} \right)$ for $a \in E^{\times}$ representing a coset of $E^{\times 2}$. We write $(\mu \tau)^{a}$ for the conjugate representation of $\mu \tau$ by the element $g(a)$. Since the quadratic Hilbert symbol is non-degenerate, if $a \in E^{\times} -E^{\times 2}$ then $\mu \neq \mu^{a}$ where $\mu^{a}(\tilde{z}) = \mu(\tilde{z}) (a, z)$ with $z=p(\tilde{z})$. It follows that if $a \in E^{\times} - E^{\times 2}$ then $\mu \tau \ncong (\mu \tau)^{a}$, indeed the central characters $\mu \tau$ and $(\mu \tau)^{a}$ are different. By Clifford theory, $\tilde{\pi} := \ind_{\widetilde{\GL}_{2}(E)_{+}}^{\widetilde{\GL}_{2}(E)} (\mu \tau)$ is an irreducible admissible genuine representation of $\widetilde{\GL}_{2}(E)$. Moreover, every irreducible admissible genuine representation of $\widetilde{\GL}_{2}(E)$ arises in this fashion. If $\tilde{\pi} = \ind_{\widetilde{\GL}_{2}(E)_{+}}^{\widetilde{\GL}_{2}(E)} (\mu \tau)$ then by Mackey theory it is easy to see that 
\begin{equation} \label{res:GL to GL+}
\tilde{\pi}|_{\widetilde{\GL}_{2}(E)_{+}} = \bigoplus_{a \in E^{\times}/E^{\times 2}} (\mu \tau)^{a},
\end{equation}
Since $(\mu \tau)^{a} \ncong (\mu \tau)^{b}$ if $ab^{-1} \notin E^{\times 2}$, the restriction of $\tilde{\pi}$ to $\widetilde{\GL}_{2}(E)_{+}$ is multiplicity free. Further restriction of $\tilde{\pi}$ to $\widetilde{\SL}_{2}(E)$ is given by
\begin{equation}
\tilde{\pi}|_{\widetilde{\SL}_{2}(E)} = \bigoplus_{a \in E^{\times}/E^{\times 2}} \tau^{a}.
\end{equation}
This identification of $\tilde{\pi}$ restricted to $\widetilde{\SL}_{2}(E)$ allows us to study the multiplicity of the restriction of a representation of $\widetilde{\GL}_{2}(E)$ restricted to $\widetilde{\SL}_{2}(E)$, and in particular shows that it may be greater than one.

\section{$\theta$-correspondence and Waldspurger involution}
In this section, we recall some results of Waldspurger from \cite{Wald91}, related to $\theta$-correspondence between $\widetilde{\SL}_{2}(E)$ and ${\rm PGL}_{2}(E) = {\rm SO}(2,1)$ and that between $\widetilde{\SL}_{2}(E)$ and $\PD^{\times} = {\rm SO}(3)$, where ${\rm D}$ is the unique quaternion division algebra over $E$. We will use these results repeatedly. \\

Now fix a non-trivial additive character $\psi$ of $E$. With respect to this $\psi$, one has the $\theta$-correspondence between irreducible admissible genuine representations of $\widetilde{\SL}_{2}(E)$ and irreducible admissible representations of ${\rm PGL}_{2}(E)$ 
\[ \xymatrix{
{\rm Irr}(\widetilde{\SL}_{2}(E)) \ar[r]^{\theta(-, \psi)} & {\rm Irr}({\rm PGL}_{2}(E)) },
\]
as well as one between irreducible admissible genuine representations of $\widetilde{\SL}_{2}(E)$ and irreducible admissible representations of $\PD^{\times}$ 
\[ \xymatrix{
{\rm Irr}(\widetilde{\SL}_{2}(E)) \ar[r]^{\theta(-,\psi)} & {\rm Irr}(\PD^{\times}).  }
\]
This correspondence $\tau \mapsto \theta(\tau,\psi)$ depends on $\psi$ and will be abbreviated to $\tau \mapsto \theta(\tau)$ as $\psi$ will be fixed. The $\theta$-correspondence between $\widetilde{\SL}_{2}(E)$ and ${\rm PGL}_{2}(E)$ gives a one to one mapping from the subset of irreducible admissible genuine representations of $\widetilde{\SL}_{2}(E)$ which have $\psi$-Whittaker model onto all irreducible admissible representations of ${\rm PGL}_{2}(E)$. Similarly, $\theta$-correspondence between $\widetilde{\SL}_{2}(E)$ and $\PD^{\times}$ gives a one to one mapping from the subset of irreducible admissible genuine representations of $\widetilde{\SL}_{2}(E)$ which do not have $\psi$-Whittaker model onto all irreducible representations of $\PD^{\times}$. Thus $\theta$-correspondence defines a bijection (which depends on the choice of $\psi$):
\begin{equation}
{\rm Irr}(\widetilde{\SL}_{2}(E)) \longleftrightarrow {\rm Irr}({\rm PGL}_{2}(E)) \, \bigsqcup \, {\rm Irr}(\PD^{\times}).
\end{equation}
Now we can describe the Waldspurger involution \cite{Wald91} $W : {\rm Irr}(\widetilde{\SL}_{2}(E)) \rightarrow {\rm Irr}(\widetilde{\SL}_{2}(E))$ which is defined using 
\begin{enumerate} 
\item the $\theta$-correspondence from $\widetilde{\SL}_{2}(E)$ to ${\rm PGL}_{2}(E)$, 
\item the $\theta$-correspondence from $\widetilde{\SL}_{2}(E)$ to $\PD^{\times}$ and 
\item the Jacquet-Langlands correspondence between representations of ${\rm PGL}_{2}(E)$ and $\PD^{\times}$, and makes the following diagram commutative:
\end{enumerate} 
\[
\xymatrix{
\widetilde{\SL}_{2}(E) \ar@{<->}[d]_{W} \ar[r]^{\theta} & {\rm PGL}_{2}(E) \ar@{<->}[d]^{J-L} \\
\widetilde{\SL}_{2}(E) \ar[r]^{\theta} & \PD^{\times} 
} \label{diagram:involution}
\]
This involution is defined on the set of all representations of $\widetilde{\SL}_{2}(E)$ whose fixed points are precisely the irreducible admissible genuine representations which are not discrete series representations. Denote this involution by $\tau \mapsto \tau_{W}$. This involution is independent of the character $\psi$ chosen to define it.\\

The following theorem summarizes some of the results of Waldspurger from \cite{Wald91} which are relevant to our analysis. This theorem is in terms of the local $\epsilon$-factors of Jacquet-Langlands, which we will use without reviewing.	

\begin{theorem} \label{theorem:A}
Let $\tau$ be an irreducible admissible genuine representation of $\widetilde{\SL}_{2}(E)$. Let $\psi$ be a non-trivial additive character of $E$. For $a \in E^{\times}$, let $\chi_{a}$ be the quadratic character of $E^{\times}$ defined by $\chi_{a}(x)=(a,x)$ where $(-,-)$ denotes the Hilbert symbol with values in $\{ \pm 1 \}$. Both the representations $\tau$ and $\tau^{a}$ of $\widetilde{\SL}_{2}(E)$ are in the domain of theta correspondence (with respect to the character $\psi$) either with ${\rm PGL}_{2}(E)$ or with $\PD^{\times}$ if and only if 
\[
\epsilon(\theta(\tau) \otimes \chi_{a}) = \chi_{a}(-1) \epsilon(\theta(\tau)),
\]
 and then 
 \[
 \theta(\tau^{a}) \cong \theta(\tau) \otimes \chi_{a}.
 \]
If $\epsilon(\theta(\tau) \otimes \chi_{a}) = - \chi_{a}(-1) \epsilon(\theta(\tau))$, and  if $\theta(\tau)$ is a representation of ${\rm PGL}_{2}(E)$ then $\theta(\tau^{a})$ is a representation of $\PD^{\times}$ and vice-versa, and  
\[
\theta(\tau^{a}) = \theta(\tau)^{JL} \otimes \chi_{a}.
\]
\end{theorem}

\section{Multiplicity formula on restriction from $\widetilde{\GL}_{2}(E)$ to $\widetilde{\SL}_{2}(E)$} \label{higher multiplicity}
Let $\tilde{\pi}$ be an irreducible admissible genuine representation of $\widetilde{\GL}_{2}(E)$. Let $\mu$ be a character of $\tilde{Z}$ and $\tau$ an irreducible representation of $\widetilde{\SL}_{2}(E)$, which are compatible, such that $\mu\tau$ appears in $\tilde{\pi}$ restricted to $\widetilde{\GL}_{2}(E)_{+}$. We have
 \[
 \tilde{\pi}|_{\widetilde{\GL}_{2}(E)_{+}} = \bigoplus_{a \in E^{\times}/E^{\times 2}} (\mu^{a} \tau^{a})
 \]
 where $a \in E^{\times}/E^{\times 2}$ are elements of the split torus $T \cong E^{\times} \times E^{\times}$ of the form $diag(a,1)$. Since the restriction of $\mu \tau$ from $\widetilde{\GL}_{2}(E)_{+}$ to $\widetilde{\SL}_{2}(E)$ is $\tau$,  the multiplicity with which the representation $\tau$ appears in $\tilde{\pi}$, to be denoted by $m(\tilde{\pi}, \tau)$, is given by
 \[
m(\tilde{\pi}, \tau) = \# \{ a \in E^{\times}/E^{\times 2} : \tau^{a} \cong \tau \}.
 \]
\begin{lemma} \label{lemma:A}
For an irreducible admissible representation $\tau$ of $\widetilde{\SL}_{2}(E)$, and $a \in E^{\times}$, we have
\[
\tau \cong \tau^{a} \Longleftrightarrow \left\{ \begin{array}{lrl}
     (1) & \theta(\tau) \otimes \chi_{a} & \cong  \theta(\tau) \\
     (2) & \chi_{a}(-1) & =  1.
                                                                         \end{array}
                                                                         \right.
\]
\end{lemma}
\begin{proof}
It $\tau \cong \tau^{a}$, then considering the central characters on both sides, we find that $\chi_{a}(-1) =1$. Further, if $\tau \cong \tau^{a}$, then in particular, they both have $\theta$ lifts either to ${\rm PGL}_{2}(E)$ or $\PD^{\times}$, and $\theta(\tau) \cong \theta(\tau^{a})$. Thus from Theorem \ref{theorem:A} due to Waldspurger, we deduce the assertion in the lemma.
\end{proof}
\begin{corollary} \label{corollary:A2} The multiplicity of $\tau$ in $\tilde{\pi}$ is given by
\[
m(\tilde{\pi}, \tau) = \# \left\{ a \in E^{\times}/E^{\times 2} : \theta(\tau) \otimes \chi_{a} \cong \theta(\tau) \text{ and } \chi_{a}(-1)=+1 \right\}.
\]
\end{corollary}
\noindent It is well-known that for a representation $\pi$ of $\GL_{2}(E)$, cf. \cite{LL79}
\[
m(\pi) := \# \{ a \in E^{\times}/E^{\times 2} : \pi \cong \pi \otimes \chi_{a} \} \in \{1, 2, 4 \}.
\]
The condition $\chi_{a}(-1) = 1$ is automatic in some situations for example if $ -1 \in E^{\times 2}$. Thus we get $m(\tilde{\pi}, \tau)$ to be any of the following possibilities:
\[
m(\tilde{\pi}, \tau) = 1,2 \text{ or } 4
\]
for some $p$-adic field for any $p$, including $p=2$.

\section{A lemma on Waldspurger involution}
 We recall that for an irreducible admissible genuine representation $\tau$ of $\widetilde{\SL}_{2}(E)$, the central characters of $\tau$ and $\tau_{W}$ are different. The group $\GL_{2}(E)$, or what amounts to simply $E^{\times}$ sitting inside $\GL_{2}(E)$ as $\left\{ \left( \begin{matrix}e & 0 \\ 0 & 1 \end{matrix} \right) : e \in E^{\times} \right\}$, acts on the set of irreducible representations of $\widetilde{\SL}_{2}(E)$ denoted by $\tau \mapsto \tau^{a}$  for $a \in E^{\times}$. Since a similar action produces an $L$-packet for $\SL_{2}(E)$, whereas for $\widetilde{\SL}_{2}(E)$, one defines an $L$-packet by taking $\tau$ and $\tau_{W}$, we investigate in this section if it can happen that $\tau_{W} \cong \tau^{a}$ for some $ a \in E^{\times}$ and $\tau$ a discrete series representation of $\widetilde{\SL}_{2}(E)$.
\begin{lemma} \label{lemma:B}
Let $\tau$ be a discrete series representation of $\widetilde{\SL}_{2}(E)$. Let $\psi$ be a non-trivial additive character of $E$ such that $\tau$ has $\theta$ lift to ${\rm PGL}_{2}(E)$ with respect to $\psi$. Then there exists $a \in E^{\times}$ with $\tau^{a} \cong \tau_{W}$ if and only if for $\pi=\theta(\tau, \psi)$, we have 
\begin{enumerate}
\item[(i)] $\pi \cong \pi \otimes \chi_{a}$
\item[(ii)] $\chi_{a}(-1)=-1$.
\end{enumerate}
\end{lemma}
\begin{proof}
Let $\pi = \theta(\tau, \psi)$ and $\theta(\tau_{W}, \psi)= \pi^{JL}$, where $\pi^{JL}$ denotes the representation of $\PD^{\times}$ which is associated to $\pi$ via the Jacquet-Langlands correspondence. From Theorem \ref{theorem:A} it follows that if $\epsilon(\pi \otimes \chi_{a}) = \chi_{a}(-1) \epsilon(\pi)$, then $\tau^{a}$ lift to ${\rm PGL}_{2}(E)$ and not to $\PD^{\times}$ and hence $\tau^{a}$ cannot be isomorphic to $\tau_{W}$. Thus if $\tau^{a}$ were isomorphic to $\tau_{W}$, then we must have $\epsilon(\pi \otimes \chi_{a}) = - \chi_{a}(-1) \epsilon(\pi)$. In this case, by Theorem \ref{theorem:A}, $\tau^{a}$ lifts to $\PD^{\times}$ and is $\pi^{JL} \otimes \chi_{a}$. Therefore 
\[
\tau^{a} \cong \tau_{W} \Longleftrightarrow \left\{ \begin{array}{lrcl}
(i) & \epsilon(\pi \otimes \chi_{a}) &=& - \chi_{a}(-1) \epsilon(\pi) \\
(ii) & \pi^{JL} &\cong& \pi^{JL} \otimes \chi_{a}.
\end{array} \right.
\] 
The equations (i) and (ii) can be combined to say that 
\[
\tau^{a} \cong \tau_{W} \Longleftrightarrow \left\{ \begin{array}{lrl}
 (i) & \pi & \cong \pi \otimes \chi_{a} \\
 (ii) & \chi_{a}(-1) & = -1. 
 \end{array}
 \right. 
\]
This completes the proof of the lemma.
\end{proof}
As a consequence of Lemma \ref{lemma:A} and Lemma \ref{lemma:B}, we obtain:
\begin{corollary}
Let $\tau$ be an irreducible genuine discrete series representation of $\widetilde{\SL}_{2}(E)$. Let $m_{1} = \# \{ \tau^{a}, (\tau_{W})^{a} \mid a \in E^{\times} \}$, and let $m_{2}$ be the cardinality of the $L$-packet of $\SL_{2}(E)$ determined by $\theta(\tau, \psi)$. Then
\[
m_{1} \cdot m_{2} = 2 [E^{\times} : E^{\times 2}].
\]
\end{corollary}

\begin{corollary}
If $\pi$ is a principal series representation of ${\rm PGL}_{2}(E)$ with $\pi \otimes \chi_{a} \cong \pi$, then $\pi$ must be the principal series representation $Ps(\mu, \mu \chi_{a})$ with $\mu^{2} = \chi_{a}$, and as a result $\chi_{a}(-1) = \mu^{2}(-1) = 1$. Therefore if $\tau$ is an irreducible admissible representation of $\widetilde{\SL}_{2}(E)$ with $\theta(\tau)$ an irreducible principal series representation of ${\rm PGL}_{2}(E)$, then for $m_{1} = \# \{ \tau^{a} \mid a \in E^{\times} \}$, and $m_{2}$ the cardinality of the $L$-packet of $\SL_{2}(E)$ determined by $\theta(\tau, \psi)$,
\[
m_{1} \cdot m_{2} = [E^{\times} : E^{\times 2}].
\]
\end{corollary}

\end{document}